\documentclass[11pt,twoside]{article} 
\usepackage{bbm}
\usepackage{amsmath,amssymb}
\makeatletter
\newcommand*{\mint}[1]{%
  \mint@l{#1}{}%
}
\newcommand*{\mint@l}[2]{%
  \@ifnextchar\limits{%
    \mint@l{#1}%
  }{%
    \@ifnextchar\nolimits{%
      \mint@l{#1}%
    }{%
      \@ifnextchar\displaylimits{%
        \mint@l{#1}%
      }{%
        \mint@s{#2}{#1}%
      }%
    }%
  }%
}
\newcommand*{\mint@s}[2]{%
  \@ifnextchar_{%
    \mint@sub{#1}{#2}%
  }{%
    \@ifnextchar^{%
      \mint@sup{#1}{#2}%
    }{%
      \mint@{#1}{#2}{}{}%
    }%
  }%
}
\def\mint@sub#1#2_#3{%
  \@ifnextchar^{%
    \mint@sub@sup{#1}{#2}{#3}%
  }{%
    \mint@{#1}{#2}{#3}{}%
  }%
}
\def\mint@sup#1#2^#3{%
  \@ifnextchar_{%
    \mint@sup@sub{#1}{#2}{#3}%
  }{%
    \mint@{#1}{#2}{}{#3}%
  }%
}
\def\mint@sub@sup#1#2#3^#4{%
  \mint@{#1}{#2}{#3}{#4}%
}
\def\mint@sup@sub#1#2#3_#4{%
  \mint@{#1}{#2}{#4}{#3}%
}
\newcommand*{\mint@}[4]{%
  \mathop{}%
  \mkern-\thinmuskip
  \mathchoice{%
    \mint@@{#1}{#2}{#3}{#4}%
        \displaystyle\textstyle\scriptstyle
  }{%
    \mint@@{#1}{#2}{#3}{#4}%
        \textstyle\scriptstyle\scriptstyle
  }{%
    \mint@@{#1}{#2}{#3}{#4}%
        \scriptstyle\scriptscriptstyle\scriptscriptstyle
  }{%
    \mint@@{#1}{#2}{#3}{#4}%
        \scriptscriptstyle\scriptscriptstyle\scriptscriptstyle
  }%
  \mkern-\thinmuskip
  \int#1%
  \ifx\\#3\\\else_{#3}\fi
  \ifx\\#4\\\else^{#4}\fi
}
\newcommand*{\mint@@}[7]{%
  \begingroup
    \sbox0{$#5\int\m@th$}%
    \sbox2{$#5\int_{}\m@th$}%
    \dimen2=\wd0 %
    \let\mint@limits=#1\relax
    \ifx\mint@limits\relax
      \sbox4{$#5\int_{\kern1sp}^{\kern1sp}\m@th$}%
      \ifdim\wd4>\wd2 %
        \let\mint@limits=\nolimits
      \else
        \let\mint@limits=\limits
      \fi
    \fi
    \ifx\mint@limits\displaylimits
      \ifx#5\displaystyle
        \let\mint@limits=\limits
      \fi
    \fi
    \ifx\mint@limits\limits
      \sbox0{$#7#3\m@th$}%
      \sbox2{$#7#4\m@th$}%
      \ifdim\wd0>\dimen2 %
        \dimen2=\wd0 %
      \fi
      \ifdim\wd2>\dimen2 %
        \dimen2=\wd2 %
      \fi
    \fi
    \rlap{%
      $#5%
        \vcenter{%
          \hbox to\dimen2{%
            \hss
            $#6{#2}\m@th$%
            \hss
          }%
        }%
      $%
    }%
  \endgroup
}
\usepackage{mathrsfs}
\usepackage{amssymb}
\usepackage{amsmath}
\usepackage{amsthm}
\usepackage{amsfonts}
\usepackage{color}
\usepackage{graphicx}
\usepackage[active]{srcltx}
\usepackage{tikz}
\usepackage{pgflibraryarrows}
\usepackage{pgflibrarysnakes}

\usepackage[cp1252]{inputenc}

\usepackage{mathrsfs}
\usepackage{graphicx}

\usepackage[active]{srcltx}

\allowdisplaybreaks

\usepackage{titletoc}
\titlecontents{section}[0pt]{\addvspace{2pt}\filright}
              {\contentspush{\thecontentslabel\ }}
              {}{\titlerule*[8pt]{.}\contentspage}

\def\rr{{\mathbb R}}

\def\fz{\infty}
\def\az{\alpha}

\def\dist{{\mathop\mathrm{\,dist\,}}}
\def\loc{{\mathop\mathrm{\,loc\,}}}

\def\lz{\lambda}
\def\dz{\delta}
\def\bdz{\Delta}
\def\ez{\epsilon}

\def\gz{{\gamma}}

\def\bint{{\ifinner\rlap{\bf\kern.35em--}
\int\else\rlap{\bf\kern.45em--}\int\fi}\ignorespaces}

\def\bbint{{\ifinner\rlap{\bf\kern.35em--}
\hspace{0.078cm}\int\else\rlap{\bf\kern.45em--}\int\fi}\ignorespaces}

\newtheorem{thm}{Theorem}[section]
\newtheorem{lem}[thm]{Lemma}
\newtheorem{rem}[thm]{Remark}
\numberwithin{equation}{section}

\title
{\Large\bf The boundedness of stable solutions to semilinear elliptic
equations with linear lower bound on nonlinearities
\footnotetext{\hspace{-0.35cm}
\endgraf
 2020 {\it Mathematics Subject Classification:} 35J61, 35B65, 35B35.
 \endgraf {\it Key words and phrases:}  semilinear elliptic equation, stable solution, H\"older regularity
\endgraf The author is supported by National Natural Science 
Foundation of China (No. 12201612 \& No. 11688101) and  Project  funded  by China Postdoctoral Science Foundation(No. BX20220328). The author is also supported by National
 key R \& D Program of China (No. 2021 YFA 1003100).}}
\author{Fa Peng}
 \date{\today}
\begin{document}

\arraycolsep=1pt
\allowdisplaybreaks
 \maketitle

\begin{center}
\begin{minipage}{13.5cm}\small
 \noindent{\bf Abstract.}\quad
Let $2\le n\le9$.
Suppose that  $f:\rr\to \rr$ is locally Lipschitz function satisfying
$f(t)\ge A\min\{0,t\}-K$ for all $t\in \rr$ with some constant $A\ge0$ and $K\ge 0$. We establish
an a priori
interior H\"older regularity of $C^2$-stable solutions to the semilinear elliptic equation
$-\bdz u=f(u)$. If, in addition, $f$ is nondecreasing and convex,
we obtain the interior H\"older regularity of $W^{1,2}$-stable solutions. Note that the dimension $n\le9$ is optimal.

\end{minipage}
\end{center}

\section{Introduction}
Let $\Omega$ be a bounded domain in $\rr^n$ with $n\ge2$. We are considered with the
semilinear elliptic equation
\begin{equation}\label{eq}
-\bdz u=f(u)\quad{\rm in}\quad \Omega,
\end{equation}
where the  nonlinearity $f:\rr\to \rr$ is locally Lipschitz in $\rr$ (for short
$f\in{\rm Lip}_{\loc}(\rr)$). Recall that, we say that $u
:\Omega\to \rr$ is a
$W^{1,2}_{\loc}$-weak solution to equation \eqref{eq} if $u\in W^{1,2}_{\loc}(\Omega)$, $f(u)\in L^1_{\loc}(\Omega)$ and
\begin{equation}\label{weak}
\int_{\Omega}Du\cdot D\xi\,dx-\int_{\Omega}f(u)\xi\,dx=0\quad \forall \xi\in C^\fz_c(\Omega).
\end{equation}
Moreover,  a $W^{1,2}_{\loc}$-weak solution $u$ is called as a stable solution if
$f'_{-}(u)\in L^1_{\loc}(\Omega)$ and
\begin{equation}\label{st}
\int_{\Omega}f'_{-}(u)\xi^2\,dx\le \int_{\Omega}|D\xi|^2\,dx\quad \forall \xi\in C^\fz_c(\Omega),
\end{equation}
where we write
$$f'_{-}(t):=\liminf_{h\to 0}\frac{f(t+h)-f(t)}{h},\quad \forall t\in \rr.$$
Notice that $f'_{-}(t)=f'(t)$ whenever $f\in C^1(\rr)$.

In 1975, Crandall-Rabinowitz \cite{cr75} initiated the study of the regularity of
stable solutions for the exponential and power-type nonlinearities when $n\le 9$.
In general, Brezis in \cite{b03}
asked an open problem for boundedness of stable solutions when $n\le 9$ to
a large class of nonlinearities.
The dimension $n\le9$ is optimal to get the boundedness of
stable solutions, since Joseph-Lundgren \cite{jl73} showed that
$-2\ln|x|\in W^{1,2}_0(B_1)\backslash L^\fz(B_1)$ is a
stable solution to $-\bdz u=2(n-2)e^u$ in $B_1$ when $n\ge 10$.

Towards this open problem,  there have been great interests  to study the
boundedness of stable solutions.
The boundedness of stable solutions was proved by Nedev \cite{n20} for $n\le 3$, and by
Cabr\'e \cite{c10} for $n=4$ when $f\in C^1(\rr)$ is nondecreasing, convex and nonnegative. Recently,  through the delicate
compactness method, Cabr\'e, Figalli, Ros-Oton and Serra \cite{cfrs} obtained the boundedness of stable solutions up to
 the optimal dimension $n=9$ when $f\in {\rm Lip}_{\loc}(\rr)$ is nondecreasing, convex and nonnegative and hence
 they completely solved the Brezis' problem \cite{b03}; see also a quantitative proof by
 Cabr\'e \cite{c22,c23}. The key point is that, if
 $f\in {\rm Lip}_{\loc}(\rr)$ is nonnegative,  Cabr\'e, Figalli, Ros-Oton and Serra
 \cite[Theorem 1.2]{cfrs}
 established the following a  priori $C^{0,\az}$-estimates for all
 $C^2$-stable solutions which are bounded by a $L^1$-norm of $u$:
\begin{align}\label{ho}
\|u\|_{C^{0,\az}(B_{1/2})}\le C(n)\|u\|_{L^1(B_1)}
\end{align}
for some dimensional constant $\az\in (0,1)$. Moreover, if $f$ is nondecreasing and convex in addition, they also showed that the $C^{0,\az}(\overline \Omega)$-norm of $u$ is bounded by the $L^1(\Omega)$-norm of $u$ in a bounded $C^3$ domain $\Omega\subset \rr^n$, with $u=0$ on $\partial \Omega$. It is worth noting that  the new quantitative proof of boundary H\"older regularity is given by Cabr\'e \cite{c22,c23}. Recently,
Erneta \cite{e23} improves the boundary result of \cite{cfrs} to
include $C^{1,1}$ domains, instead of $C^3$ domains. The $C^{0,\az}$-bound via $L^1$-norm is crucial to get the boundedness of extremal solution; see
\cite{cfrs,c22}. Throughout this paper, we denote by  $B_r(y)$ the ball of radius of
$r$ centered at $y$ and simply by $B_r$ in case the ball is centered at the origin; we also denote by $C(a,b,\cdots)$ a positive constant
depending only on the parameters $a,b, \cdots$ whose value may change line
to line. For any function $v\in W^{1,2}(B_r)$, the norm $\|v\|_{W^{1,2}(B_r)}$ stands for
the norm of $W^{1,2}$-Sobolev space.

 Note that, the assumption
$f\ge 0$ is needed to get the  interior H\"older estimate \eqref{ho} when $n\le9$.
Cabr\'e \cite{c21,c22} asked
a question if an interior $C^{0,\az}$-estimate
could hold for $n\le 9$ without the hypothesis $f\ge 0$.
If the stable solution is radial, the interior $L^\fz$-estimate of stable
solutions was got by Cabr\'e-Capella \cite{cc06} for
all $f\in C^1(\rr)$. For nonradial case, it is well-known that the interior $L^\fz$-estimate of stable
solutions holds for all $f\in C^1(\rr)$  by Cabr\'e \cite{c10,c19} when $n\le 4$.
When $n=5$, for all $f\in {\rm Lip}_{\loc}(\rr)$ the H\"older estimates \eqref{ho}
with the right-hand side replaced
by the $L^2$-norm of $Du$  was established in \cite{pzz22}.
 When $n\le 9$ and $f\ge -K$ for some
constant $K\ge 0$, Cabr\'e \cite{c22,c23} showed that
the following interior $C^{0,\az}$-estimate via a quantitative method:
\begin{align}\label{ho1}
\|u\|_{C^{0,\az}(B_{1/2})}\le C(n)(\|u\|_{L^1(B_1)}+K).
\end{align}

In this paper,  under the assumption:
\begin{align}\label{as-f}
f(t)\ge A\min\{0,t\}-K,\quad \forall t\in \rr,\quad {\rm for\ some\ constant\
}\ A\ge0,K\ge 0,
\end{align}
we prove that the following interior $C^{0,\az}$-regularity of the stable solution to equation \eqref{eq}.
\begin{thm}\label{th1}
Let $2\le n\le 9$, suppose that $f\in {\rm Lip}_{\loc}(\rr)$ satisfies \eqref{as-f}
and let $u\in C^2(\Omega)$ be a stable solution to \eqref{eq}.
Then
\begin{align}\label{holder}
\underset{{B_r(x)}}{\rm osc}\ u\le C\left(\frac{r}R\right)^{\alpha}(\|u\|_{W^{1,2}(B_{2R}(x))}+1),\quad
\forall x\in\Omega,\, 0<r<R<\frac1 4\dist(x,\partial\Omega),
\end{align}
where $\az=\az(n)\in (0,1)$ and $C=C(n,A,K)$.
\end{thm}

Thanks to Theorem \ref{th1}, if
$f$ is nondecreasing and convex additionally, we
obtain the interior H\"older regularity of $W^{1,2}$-stable solution
for $n\le 9$ by applying the argument of \cite[Theorem 1.2]{pzz22}; we also refer to
  \cite[Proposition 4.2]{cfrs} under the assumption $f\ge 0$.
\begin{thm}\label{th2}
Let $2\le n\le9$ and let $f\in {\rm Lip}_{\loc}(\rr)$ satisfy \eqref{as-f}. Suppose that  $f$ is convex and nondecreasing. If $u\in W^{1,2}(\Omega)$ is a stable solution
to equation \eqref{eq}, then \eqref{holder} also holds.
\end{thm}

\begin{rem}\label{rem}
\rm
Recall that, we say that $u
:\Omega\to \rr$ is a
$L^1_{\loc}$-weak solution to equation \eqref{eq} if $u\in L^1_{\loc}(\Omega)$, $f(u)\in L^1_{\loc}(\Omega)$ and
\begin{equation*}
-\int_{\Omega}u \bdz\xi\,dx=\int_{\Omega}f(u)\xi\,dx\quad \forall \xi\in C^\fz_c(\Omega).
\end{equation*}
A $L^1_{\loc}(\Omega)$-weak solution $u$ to \eqref{eq} is  called as  a $L^1_{\loc}$-stable solution  if $u$ satisfies \eqref{st} for $f\in {\rm Lip}_{\loc}(\rr)$.
When dimension $n\ge 3$, Theorem \ref{th2} is not correct for some $L^1_{\loc}$-stable solutions.
Indeed, Brezis and V\'azquez \cite{bv97} showed that,
when $n\ge 3$, $|x|^{-2/(p-1)}-1$ is a $L^1_{\loc}$-stable
solution to equation
$$-\bdz v=\frac 2{p-1}\left(n-\frac {2p}{p-1}\right)(1+v)^p\quad{\rm in}\quad B_1\
\ {\rm whenever}\ \frac {n}{n-2}<p\le \frac{n+2\sqrt{n-1}}{n-4+2\sqrt{n-1}},$$
while $|x|^{-2/(p-1)}-1$ does not belong to $L^\fz(B_1)$ and $W^{1,2}(B_1)$.
When dimension $n=2$, Villegas \cite[Theorem 1.3]{v16} showed that all the
 $L^1(B_1)$-radial stable solutions are bounded for any
 nonlinearity $f\in C^1(\rr)$. However, in the nonradial setting, it remains unclear whether  planar
 $L^1_{\loc}(\Omega)$-stable solutions are bounded for  every
 nonlinearity $f$.

\end{rem}

In Theorem \ref{th2}, we need to approximate the $W^{1,2}$-stable solution via
$C^2$-stable solution. First, we consider the  Dirichlet problem $-\bdz v=\bar Av-K$ in
$B_{r_0}(x_0)\Subset \Omega$ with $v=u$ on $\partial B_{r_0}(x_0)$ where $\bar A=\min\{f'_{-}(0),A\}$ and
$r_0$ depend on $n$, ${\rm dist}(x_0,\partial \Omega)$ and
$A$. Then one can adapt arguments for \cite{cfrs,pzz22} to
build $C^2$-stable solution $u^{\ez}$ to $-\bdz u^{\ez}=f_{\ez}(u^{\ez})$ in $B_{r_0}(x_0)$ with
$u^{\ez}=u$ on $\partial B_{r_0}(x_0)$.

Let $2\le n\le 9$. Since Theorem \ref{th2} follows from Theorem \ref{th1} by applying the
arguments for \cite{cfrs,pzz22}, the key is to
prove Theorem \ref{th1}. By a covering and scaling argument, it suffices to prove \begin{align}\label{du-mo}
\int_{B_{\rho}}(Du\cdot x)^2|x|^{-n}\,dx\le C\rho^{2\az}, \quad
\ \ \forall 0<\rho<1/2,
\end{align}
for some constant $C=C(A,K,n)(1+\|u\|^2_{W^{1,2}(B_1)})$, which, together with the  Morrey's estimates for
radial derivatives in \cite{cfrs,c22}, gives Theorem \ref{th1}.

Towards \eqref{du-mo}, we have to improve the compactness argument used by
Cabr\'e, Figalli, Ros-Oton and Serra \cite{cfrs} when $f\ge 0$. To illustrate our
 improvement clearly, we first summarize their original idea as follows:
\begin{enumerate}
\item[$\bullet$] {\bf Step 1.} First, by choosing a  test function
$\xi=(x\cdot Du)|x|^{-(n-2)/2}\eta$
in stability inequality \eqref{st} for some smooth cut-off function $\eta$, they
showed the crucial inequality \eqref{keq-2}; see Lemma \ref{key-le} in the appendix also in
\cite[Lemma 2.1]{cfrs}.
Note that the inequality \eqref{keq-2} holds for all $f\in {\rm Lip}_{\loc}(\rr)$.

\item[$\bullet$] {\bf Step 2.} When $f\ge 0$, they employed a compactness
argument to control the  $L^2$-norm  of $Du$ by the $L^2$-norm of radial derivative of $u$ in an annulus; see \cite[Lemma 3.1]{cfrs}.
The crucial fact here is that the proof of \cite[Lemma 3.1]{cfrs}
 $f\ge0$ is needed.

\item[$\bullet$] {\bf Step 3.} Thanks to \cite[Lemmas 2.1 and 3.1]{cfrs}, by a suitable
iteration  result(see \cite[Lemma 3.2]{cfrs}) one concludes \eqref{du-mo}.
\end{enumerate}

In Step 2 above, note that $f\ge 0$ is fully used.
Indeed, the compactness argument as in \cite[Lemma 3.1]{cfrs} is
based on the $W^{1,1}\cap L^{2+\gz}$-estimates of $Du$
and the non-existence of nontrivial 0-homogeneous for superharmonic functions.
More precisely, given a sequence of stable solution $u_k$ to
$-\bdz u_k=f_k(u_k)$,  by normalizing one may
assume that $\|Du_k\|_{L^2(B_1)}=1$. Then, this leads to the uniform
$W^{2,1}\cap W^{1,2+\gz}$-estimates of $u_k$ by a doubling assumption
\eqref{dou} on $Du_k$ and a priori $W^{1,1}\cap L^{2+\gz}$-estimates of
$Du_k$. However, under the assumption
\eqref{as-f}, $f$ is not necessarily nonnegative. Thus we can
not simply follow their argument to get same result as in \cite[Lemma 3.1]{cfrs}.
In particular, when $f_k$ is not nonnegative, $u_k$ is not necessarily
in $W^{1,2+\gz}\cap W^{2,1}$ uniformly in $k$ and the limiting function
of $u_k$ is also not necessarily superharmonic. Hence we need some new ideas
to overcome this difficulty.

Under $f(t)\ge A\min\{0,t\}-K$, we overcome the  difficulty above via a new version of
 \cite[Lemma 3.1]{cfrs}; see Lemma \ref{key-it}. To be more precise, by introducing a lower bound assumption
\eqref{icd} on gradient of $u$, combined with a priori
$W^{1,1}\cap L^{2n/(n-1)}$-estimates of $Du$ (see Lemma \ref{sob}) and   a doubling assumption \eqref{dou} on $Du$, we obtain
the uniform $W^{1,2+\gz}\cap W^{2,1}$-
estimates of $u_k$. In particular, the assumption
\eqref{icd} also yields the superharmonic property for the
limiting function of $u_k$ (noting that $u_k$ is not necessarily superharmonic).
Then applied a compactness argument of \cite{cfrs} to $u_k$ as desired.

Moreover,  since a lower bound assumption
\eqref{icd} on $Du$ is needed in Lemma \ref{key-it}, the iteration  result as in
\cite[Lemma 3.2]{cfrs} can not directly used in proving \eqref{du-mo}.
To this end, we  built a new version of \cite[Lemma 3.2]{cfrs} which is based on
Lemma \ref{key-it};
see Lemma \ref{it}.

The paper is organized as follows. Some a priori estimates on $u$ is established in Section 2, which can be used in proof of the key lemmas.
 In Section 3, we show the key Lemmas
 \ref{key-it} and \ref{it}.
 Finally, we prove Theorems \ref{th1} and \ref{th2} in Section 4.
\begin{rem}
\rm
(i) When $f\ge -K$, Cabr\'e \cite{c22}
provided a quantitative proof to control the $L^2$-norm of $Du$ by the the $L^2$-norm of radial derivative of $u$ and $L^1$-norm of constant $K$ without a doubling assumption \eqref{dou} on $Du$.

(ii) The improved compactness argument in this paper can be extended the
 nonlinear equations $-\bdz_p u=f(u)$ involving $p$-Laplacian when
 $f(t)\ge A\min\{0,t\}-K$ for proving analogues results. We will work on this in forthcoming for
 more general lower bound on $f$. Recall that, when
 $f\ge 0$, the interior H\"older estimates of stable solutions to
 $-\bdz_p u=f(u)$ is established by  Cabr\'e-Miraglio-Sanchon
 \cite{cms} for the optimal dimension  $n<p+4p/(p-1)$ when $p>2$ and
 $n<5p$ when $p\in (1,2)$. This proof is based on
 the compactness argument of \cite{cfrs}.
\end{rem}

\section{Some a priori Sobolev and Morrey estimates}
In this section we provide some a priori estimates for $C^2$-stable solutions $u$
to equation $-\bdz u=f(u)$, which will be used in proving the key lemmas. We first show a Morrey bound of $u$ for all $f\in {\rm Lip}_{\loc}(\rr)$.
\begin{lem}\label{mo}
Let $3\le n\le 9$, $f\in {\rm Lip}_{\loc}(\rr)$, and let $u\in C^2(B_2)$ be a stable solution to
$-\bdz u=f(u)$ in $B_2$. Then for all $0<r<1$ and for all $p>2$ we have
\begin{align}\label{mo-1}
r^{-\frac{n(p-2)}{p}}\int_{B_r}u^2\,dx
\le C(n,p)\int_{B_2}(u^2+|Du|^2)\,dx.
\end{align}
\end{lem}
It is not hard to prove the Lemma \ref{mo} from the BMO estimate of
$u$ due to \cite{cfrs}(see also \cite{pzz22}). For the readers of convenience we
provide the proof in the appendix.

The next lemma gives a priori $W^{1,\frac{2n}{n-1}}$ and $W^{2,1}$-
estimate of $u$ which relies on the argument of \cite{cfrs} and Sobolev inequality.

\begin{lem}\label{sob}
Let $n\ge 2$ and let $f\in {\rm Lip}_{\loc}(\rr)$ satisfy \eqref{as-f}.
Suppose that $u\in C^2(B_{r})$ is a stable solution to
$-\bdz u=f(u)$ in
$B_r$. For all
$\eta\in C^\fz_c(B_{r})$, we have
\begin{align}\label{w1n}
\left(\int_{B_{r}}(|Du|^2\eta^2)^{\frac{n}{n-1}}\,dx\right)^{\frac{n-1}{n}}&\le C(n)
\left(\int_{B_{r}}|Du|^2|D\eta|^2
\,dx\right)^{\frac12}\left(\int_{B_{r}}|Du|^2\eta^2\,dx\right)^{\frac12}\nonumber\\
&\quad +2\int_{B_{r}}(A|u|+K)|Du|\eta^2\,dx
\end{align}
and
\begin{align}\label{w21}
\int_{B_{r}}|D^2u|\eta^2\,dx&\le C(n)
\left(\int_{B_{r}}|Du|^2|D\eta|^2
\,dx\right)^{\frac12}\left(\int_{B_{r}}\eta^2\,dx\right)^{\frac12} +2
\int_{B_{r}}(A|u|+K)\eta^2\,dx.
\end{align}
\end{lem}

\begin{proof}[Proof of Lemma \ref{sob}]
We first show \eqref{w21}. By triangle inequality, we have
\begin{align}\label{so-1}
&|D^2u|\le \left|D^2u-\frac{\bdz_\fz u}{|Du|^2}\frac{Du}{|Du|}
\otimes\frac{Du}{|Du|}\right|+\left|\frac{\bdz_\fz u}{|Du|^2}-\bdz u\right|+|\bdz u|\quad
{\rm a.e.\ in}\ B_r,
\end{align}
where $\otimes$ stands for tensor product, that is, $a\otimes b=(a_ib_j)_{1\le i,j\le n}$ for
all $a,b\in \rr^n$.
Observe that, the first term and the second term in the right-hand side of the inequality
\eqref{so-1} can be bounded by $C(n)[|D^2u|^2-|D|Du||^2]^{\frac 12}$ a.e. in $\rr^n$
(see for example \cite[Lemma 1.5]{pzz22}); the last term in the right-hand side of the inequality \eqref{so-1} can be estimated as
$$|\bdz u|\le -\bdz u+2(A|\min\{0,u\}|+K)\le -\bdz u+2(A|u|+K)$$
due to \eqref{as-f}. Therefore, \eqref{so-1} becomes
\begin{align}\label{so-2}
&|D^2u|\le C(n)[|D^2u|^2-|D|Du||^2]^{\frac 12}
-\bdz u+2(A|u|+K)\quad
{\rm a.e.\ in}\ B_r.
\end{align}
Multiplying both sides by a test function $\eta^2$ with
$\eta\in C^\fz_c(B_r)$ one has
\begin{align}\label{so-3}
&\int_{B_{r}}|D^2u|\eta^2\,dx\nonumber\\
&\le C(n)\int_{B_{r}}[|D^2u|^2-|D|Du||^2]^{\frac 12}\eta^2\,dx
-\int_{B_{\rho}}\bdz u\eta^2\,dx+2\int_{B_{\rho}}|Au-K|\eta^2\,dx.
\end{align}
Recalling that the following inequality due to Sternberg and Zumbrun \cite{sz}
\begin{align}\label{so-st}
\int_{B_{r}}[|D^2u|^2-|D|Du||^2]\eta^2\,dx\le \int_{B_r}
|Du|^2|D\eta|^2\,dx.
\end{align}
By H\"older inequality, we get
\begin{align}\label{so-4}
\int_{B_{r}}[|D^2u|^2-|D|Du||^2]^{\frac 12}\eta^2\,dx
&\le \left(\int_{B_{r}}[|D^2u|^2-|D|Du||^2]\eta^2\,dx\right)^{\frac 12}
 \left(\int_{B_{r}}\eta^2\,dx\right)^{\frac 12}\nonumber\\
 &\le \left(\int_{B_{r}}|Du|^2|D\eta|^2\,dx\right)^{\frac 12}
 \left(\int_{B_{r}}\eta^2\,dx\right)^{\frac 12}.
\end{align}
Also, via integration by parts and H\"older inequality again, we obtain
\begin{align}\label{so-5}
-\int_{B_{r}}\bdz u\eta^2\,dx=-2\int_{B_r} Du\cdot D\eta \eta\,dx
\le
\left(\int_{B_{r}}|Du|^2|D\eta|^2\,dx\right)^{\frac 12}
 \left(\int_{B_{r}}\eta^2\,dx\right)^{\frac 12}.
\end{align}
Now, inserting \eqref{so-3}, \eqref{so-4} and \eqref{so-5}, this proves
\eqref{w21}.

To get \eqref{w1n}, we claim that
\begin{align}\label{wcl}
&\int_{B_{r}}|D^2u||Du|\eta^2\,dx\nonumber\\
&\le C(n)
\left(\int_{B_{r}}|Du|^2|D\eta|^2
\,dx\right)^{\frac12}\left(\int_{B_{r}}|Du|^2\eta^2\,dx\right)^{\frac12}
+2\int_{B_{r}}|Au-K||Du|\eta^2\,dx.
\end{align}
Assume that this holds for the moment. For all $\eta\in C^\fz_c(B_{r})$, a Sobolev inequality \cite[Theorem 7.10]{gt83} gives us that
\begin{align*}
\left(\int_{B_{r}}(|Du|^2\eta^2)^{\frac{n}{n-1}}\,dx\right)^{\frac{n-1}{n}}&\le
\int_{B_{r}}|D(|Du|^2\eta^2)|\,dx,
\end{align*}
which further yields
\begin{align*}
\left(\int_{B_{r}}(|Du|^2\eta^2)^{\frac{n}{n-1}}\,dx\right)^{\frac{n-1}{n}}\le
 4\int_{B_{r}}|D^2u||Du|\eta^2\,dx+2\int_{B_{r}}|\eta||D\eta||Du|^2\,dx.
\end{align*}
From this, \eqref{w1n} follows from \eqref{wcl} and H\"older inequality with  the last term in the right-hand side of \eqref{wcl}.

To this end, multiplying both sides in \eqref{so-2} by $|Du|\eta^2$ one has
\begin{align}\label{so-6}
&\int_{B_{r}}|Du||D^2u|\eta^2\,dx\nonumber\\
&\le C(n)\int_{B_{\rho}}[|D^2u|^2-|D|Du||^2]^{\frac 12}|Du|\eta^2\,dx
-\int_{B_{r}}\bdz u|Du|\eta^2\,dx+2\int_{B_{r}}|Au-K||Du|\eta^2\,dx.
\end{align}
Applying H\"older inequality and \eqref{so-st}, the first term in the right-hand
side of \eqref{so-6} can be estimated as
\begin{align}\label{so-7}
\int_{B_{r}}[|D^2u|^2-|D|Du||^2]^{\frac 12}|Du|\eta^2\,dx
&\le \left(\int_{B_{r}}[|D^2u|^2-|D|Du||^2]\eta^2
\,dx\right)^{\frac12}\left(\int_{B_{r}}|Du|^2\eta^2\,dx\right)^{\frac12}\nonumber\\
&\le \left(\int_{B_{r}}|Du|^2|D\eta|^2
\,dx\right)^{\frac12}\left(\int_{B_{r}}|Du|^2\eta^2\,dx\right)^{\frac12}.
\end{align}
For the second term in the right-hand side of \eqref{so-6}, since
\begin{align*}
{\rm div}(|Du|Du)=|Du|\bdz u+|Du|^{-1}\bdz_\fz u=
|Du|(\frac{\bdz_\fz u}{|Du|^2}-\bdz u)+2|Du|\bdz u
\end{align*}
and noting $|\frac{\bdz_\fz u}{|Du|^2}-\bdz u|\le C(n)[|D^2u|^2-|D|Du||^2]^{\frac 12}$,
using integration by parts we deduced that
\begin{align}\label{so-8}
&-\int_{B_{r}}\bdz u|Du|\eta^2\,dx\nonumber\\
&\le C(n)\int_{B_{r}}[|D^2u|^2-|D|Du||^2]^{\frac 12}|Du|\eta^2\,dx
+\frac 12\int_{B_{r}}{\rm div}(|Du|Du)\eta^2\,dx\nonumber\\
&=C(n)\int_{B_{r}}[|D^2u|^2-|D|Du||^2]^{\frac 12}|Du|\eta^2\,dx
-\int_{B_{r}}\eta|Du|Du\cdot D\eta\,dx.
\end{align}
Note that the right-hand side in \eqref{so-8} can be bounded by the
right-hand side of \eqref{so-7} via H\"older inequality.  Hence, combing
\eqref{so-6}, \eqref{so-7} and \eqref{so-8} yields the
claim \eqref{wcl}.

\end{proof}

\section{Proof of Lemma \ref{key-it} and Lemma \ref{it}}

We begin with establishing, under a doubling assumption and some lower bound assumption  on $L^2$-norm of $Du$,
  a control of the  $L^2$-norm  of $Du$ by the $L^2$-norm of radial derivative of $u$ in an annulus.

\begin{lem}\label{key-it}
Let $f \in {\rm Lip}_{\loc}(\rr)$ satisfy \eqref{as-f}.
Suppose that $u\in C^2(B_2)$ is a stable solution to $-\bdz u=f(u)$ in $B_2$. There exists a sufficiently large constant $j\ge1$
independent of $f$ and $u$ such that the following holds:\\
Suppose that
\begin{align}\label{dou}
(2^{-j})^{-n+2}\int_{B_{2^{-j}}}|Du|^2\,dx\ge \dz (2^{-j+1})^{-n+2}\int_{B_{2^{-j+1}}}|Du|^2\,dx
\end{align}
for some $\dz>0$ and
\begin{align}\label{icd}
(2^{-j})^{-n+2}\int_{B_{2^{-j}}}|Du|^2\,dx\ge \max\left\{2^{-j},
(2^{-j})^{-n+3}
\int_{B_{2^{-j+1}}}|u|^2\,dx\right\}.
\end{align}
Then there is a  constant $C_0:=C_0(n,\dz,A,K)$ such that
\begin{align}\label{ctr0}
(2^{-j})^{-n+2}\int_{B_{3\times 2^{-j-1}}}|Du|^2\,dx\le
C_0\int_{B_{3\times 2^{-j-1}}\backslash B_{2^{-j}}}
(Du\cdot x)^2|x|^{-n}\,dx.
\end{align}

\end{lem}

\begin{proof}[Proof of Lemma \ref{key-it}]
We do this by contradiction. For convenience write $r_j=2^{-j}$.
If the conclusion of the Lemma \ref{key-it} does not hold, then we can find a
sequence $\{f_j\}_{j\ge1}$ in ${\rm Lip}_{\loc}(\rr)$ satisfying $f_j(t)\ge A\min\{0,t\}-K$ for
all $t\in \rr$ and a sequence
$\{u_j\}_{j\ge1}$ in $C^2(B_2)$ such that $u_j$ is a stable solution to
$$-\bdz u_j=f_j(u_j)\quad{\rm in}\quad B_2$$
and
\begin{align}
&r_j^{-n+2}\int_{B_{r_j}}|Du_j|^2\,dx\ge \dz (r_{j-1})^{-n+2}\int_{B_{r_{j-1}}}|Du_j|^2\,dx \label{it1},\\
&r_j^{-n+2}\int_{B_{r_j}}|Du_j|^2\,dx\ge\max\left\{r_j,
r_j^{-n+3}
\int_{B_{r_{j-1}}}|u_j|^2\,dx\right\},\label{it2}
\end{align}
while
\begin{align}\label{itx}
r_j^{-n+2}\int_{B_{3r_j/2}}|Du_j|^2\,dx> j
\int_{B_{3r_j/2}\backslash B_{r_j}}
(Du_j\cdot x)^2|x|^{-n}\,dx.
\end{align}

Now we normalize by defining
$$v_j(x):=\left(r_j^{-n+2}\int_{B_{3r_j/2}}|Du_j|^2\,dy\right)^{-1/2}
\left(u_j(r_jx)-\mint{-}_{B_{r_{j-1}}}u_j\,dy\right)\quad {\rm in}\ B_2.$$
It follows by \eqref{it1} and \eqref{itx} that
\begin{align}\label{ctr}
\int_{B_{3/2}}|Dv_j|^2\,dx=1,\quad \lim_{j\to \fz}
\int_{B_{3/2}\backslash B_1}(Dv_j\cdot x)^2\,dx=0.
\end{align}
We claim that $v_j\in W^{1, 2+\frac {2}{n-1}}(B_{7/4})\cap W^{2,1}(B_{7/4})$ uniformly in $j\ge1$ and
\begin{align}\label{low-bd}
-\int_{B_{7/4}}\bdz v_j(x)\xi\,dx \ge- C(n,\dz)(Ar_j^{1/4}+Kr_j^{3/2}
)\|\xi\|_{L^\fz(B_{7/4})}.
\end{align}
for all nonnegative function $\xi\in C^\fz_c(B_{7/4})$.

If this claim
holds for the moment,  by using the weak compactness
of Sobolev space, there exists a function
$v\in W^{1,\frac {2n}{n-2}}(B_{7/4})$ such that
$v_j\to v$ in $L^2(B_{7/4})$ and
$Dv_j\to Dv$ in $L^1(B_{7/4})$ as $j\to \fz$. Note that
$Dv_j\in L^\frac {2n}{n-2}(B_{7/4})$ uniformly $j\ge 1$. Applying
H\"older inequality we obtain
$$\|D(v_j-v)\|_{L^2(B_{7/4})}\le
\|D(v_j-v)\|^{\frac 1{n+1}}_{L^1(B_{7/4})}\|D(v_j- v)\|^{\frac{n}{n+1}}_{L^{\frac{2n}{n-1}}(B_{7/4})}
\le C\|D(v_j-v)\|^{\frac 1{n+1}}_{L^1(B_{7/4})}\to 0,$$
which shows that $Dv_j\to  Dv$ in $L^2(B_{7/4})$.
By \eqref{low-bd}, for all nonnegative function $\xi\in C^\fz_{c}
(B_{7/4})$ via integration by parts we have
\begin{align*}
-\int_{B_{7/4}}v\bdz \xi\,dx&=\lim_{j\to \fz}
-\int_{B_{7/4}}v_{j}\bdz \xi\,dx
=-\lim_{j\to \fz}\int_{B_{7/4}} \bdz v_{j} \xi\,dx\\
&\ge - C(n,\dz)\lim_{j\to \fz}(Ar_j^{1/4}+Kr_j^{3/2}
)\|\xi\|_{L^\fz(B_{7/4})}=0.
\end{align*}
Observe that $v\in W^{1,2}(B_{7/4})$. This implies that $v$ is
superharmonic in $B_{7/4}$. On the other hand, since $v_j\to v
$ in $W^{1,2}(B_{7/4})$, then by \eqref{ctr} we obtain
\begin{align}\label{ctr-1}
\quad \int_{B_{3/2}}|Dv|^2\,dx=1,\quad \int_{B_{3/2}\backslash B_1}|Dv\cdot x|^2\,dx=0.
\end{align}
This is a contradiction. Indeed, the second identity in \eqref{ctr-1}
tells us that $v$ is zero homogeneous function on $B_{3/2}\backslash B_1$.
From this, by maximum principle for superharmonic $v$ one gets
$v\equiv C$ on $B_{3/2}$ for a constant $C$. This contradicts with the first identity in \eqref{ctr-1}.

We now proof this claim.
From \eqref{it1}, one has $\int_{B_2}|Dv_j|^2\,dx\le C(n)\dz^{-1}$.
Then by Poincar\'e inequality, we also have $\|v_j\|^2_{L^2(B_2)}\le C(n)\dz^{-1}$.
Thus $v_j\in W^{1,2}(B_2)$ uniformly in $j\ge1$. Moreover, by \eqref{w1n} in Lemma \ref{w1n} we have
\begin{align}\label{xw1n}
\left(r_j^{-n+\frac {2n}{n-1}}\int_{B_{7r_{j}/4 }}|Du_j(y)|^{\frac {2n}{n-1}}\,dy\right)^{\frac n{n-1}}&\le C(n)r_j^{-n+2}\int_{B_{r_{j-1}}}|Du_j(y)|^2\,dy\nonumber\\
&+C(n)r_j^{-n+4}\int_{B_{r_{j-1}}}[A^2|u_j(y)|^2+K^2]\,dy.
\end{align}
Using the condition \eqref{it1} and \eqref{it2} we see that
\begin{align}
&r_j^2\le r_j^{-n+3}\int_{B_{r_{j-1}}}|Du_j|^2\,dx
\le  C(n)\frac 1\dz r_j^{-n+3}\int_{B_{3r_j/2}}|Du_j|^2\,dx,\label{it-1}\\
&r_j^{-n+4}\int_{B_{r_{j-1}}}|u_j|^2\,dx\le C(n)r_j^{-n+3}\int_{B_{r_{j-1}}}|Du_j|^2\,dx
\le C(n)\frac 1\dz r_j^{-n+3}\int_{B_{3r_{j}/2}}|Du_j|^2\,dx. \label{it-2}
\end{align}
Thanks this, by \eqref{xw1n} yields
\begin{align*}
\left(r_j^{-n+\frac {2n}{n-1}}\int_{B_{7r_{j}/4 }}|Du_j(y)|^{\frac {2n}{n-1}}\,dy\right)^{\frac n{n-1}}&\le C(n)\frac 1{\dz}r_j^{-n+2}\int_{B_{3r_{j}/2}}|Du_j(y)|^2\,dy\\
&+C(n)\frac 1\dz[A^2 r_j+K^2r_j^3](r_j)^{-n+2}\int_{B_{3r_{j}/2}}|Du_j(y)|^2\,dy,
\end{align*}
and  dividing both sides by
$r_j^{-n+2}\int_{B_{3r_j/2}}|D u_j|^2\,dx$ one has
\begin{align*}
\left(\int_{B_{7/4 }}|Dv_j|^{\frac {2n}{n-1}}\,dx\right)^{\frac n{n-1}}&\le C(n)\frac 1\dz
[A^2+K^2+1].
\end{align*}
This proves $v_j\in W^{1,\frac {2n}{n-1}}(B_{7/4})$ uniformly in $j\ge 1$.
Also, using \eqref{w21} in Lemma \ref{sob} we have
\begin{align*}
&r_j^{-n+2}\int_{B_{7r_{j}/4 }}|D^2u_j(y)|\,dy\\
&\le C(n)\left(r_j^{-n+2}\int_{B_{r_{j-1}}}
|Du_j(y)|^2\,dy\right)^{1/2}+C(n)\left(r_{j}^{-n+4}
\int_{B_{r_{j-1}}}
[A|u_j(y)|^2+K]\,dy\right)^{1/2},
\end{align*}
and applying $v_j$ to this inequality together with \eqref{it-1} and
\eqref{it-2} again yields
$$\int_{B_{7/4}}|D^2v_j|\,dx\le C(n)\frac1{\dz^{\frac12}}[A+K+1].$$

We finally prove \eqref{low-bd}. Since $f_j(t)\ge A\min\{0,t\}-K$ for all $t\in \rr$, a direct calculation shows that
\begin{align*}
-\bdz v_j(x)&=\left(r_j^{-n+2}\int_{B_{3r_j/2}}|Du_j|^2\,dy\right)^{-1/2}r_j^2f_j(u_j(r_jx))
\\
&\ge \left(r_j^{-n+2}\int_{B_{3r_j/2}}|Du_j|^2\,dy\right)^{-1/2}[-Ar_j^2|u_j(r_jx)|-Kr_j^2]
\quad \forall x\in B_{7/4}.
\end{align*}
For any non-negative $\xi\in C^\fz_c(B_{7/4})$, multiplying both sides by $\xi$ and integrating over on $B_{7/4}$ we obtain
\begin{align*}
-\int_{B_{7/4}}\bdz v_j(x)\xi\,dx
&\ge \left(r_j^{-n+2}\int_{B_{3r_j/2}}|Du_j|^2\,dy\right)^{-1/2}
\int_{B_{7/4}}[-Ar_j^2|u_j(r_jx)|-Kr_j^2]\xi\,dx\\
&\ge -C(n)\|\xi\|_{L^\fz(B_{7/4})}\left(r_j^{-n+2}\int_{B_{3r_j/2}}|Du_j|^2\,dy\right)^{-1/2}\\
&\quad \times
\left[A\left(r_j^{-n+4}\int_{B_{7r_j/4}}|u_j|^2\,dy\right)^{1/2}+Kr_j^2\right],
\end{align*}
where we also used H\"older inequality in last inequality. Then \eqref{low-bd} follows from
\eqref{it-1} and \eqref{it-2}. Hence we finish this proof.

\end{proof}

The following  lemma is a modified version of the result in \cite[Lemma 3.2]{cfrs}, which is a crucial point to get Theorem \ref{th1}.
\begin{lem}\label{it}
Let $\{a_j\}_{j\ge1}$, $\{b_{j}\}_{j\ge 0}$ and $\{d_j\}_{j\ge 0}$ be three sequences of nonnegative numbers satisfying $a_{j_0}\le M, b_{j_0}\le M$,$d_{j_0}\le M$,
\begin{align}\label{it-x1}
b_j\le b_{j-1},\quad a_j+b_j\le L a_{j-1},\quad d_j\le M2^{-j}\quad \mbox{for\ all}\ j\ge j_0+1,
\end{align}
and
\begin{align}\label{it-x2}
\mbox{if}\ a_j\ge \frac 12 a_{j-1}\ \mbox{ and}\  a_j\ge \max\{2^{-j},d_{j-1}\},\
\mbox{ then}\quad b_j\le L(b_{j-1}-b_j)\ \mbox{ for\ all}\ j\ge j_0+1
\end{align}
for some positive constants $M>0,L>2$ and for each fixed index $j_0\ge 1$, then there exist
constant $\theta=\theta(L)\in (1/2,1)$ and $C_0=C(L,j_0)$ such that
$$b_{j+1}\le C_0(M+1)(\theta^j+j\theta^j),\quad \forall j\ge j_0+1.$$
\end{lem}
\begin{proof}[Proof of Lemma \ref{it}]
Let $\ez>0$ be a constant to be chosen later. Define
$$\mbox{$c_j:=(a_j)^{\ez}b_j,\quad \forall j\ge j_0+1$,\quad for fixed index $j_0\ge 1$.}$$
Below we consider the following three cases.

\noindent {\it Case 1}: If $a_j\le \frac 12 a_{j-1}$, then by $b_j\le b_{j-1}$ we get
  $$c_j=(a_j)^{\ez}b_j\le 2^{-\ez}(a_{j-1})^{\ez}b_{j-1}
  \le 2^{-\ez} c_{j-1}.$$

\noindent{\it Case 2}: If $a_j\ge \frac 12 a_{j-1}$ and $a_j\ge \max\{2^{-j},d_{j-1}\}$, then
applying \eqref{it-x2} we obtain
  $$b_j\le \frac{L}{1+L}b_{j-1}.$$
Hence, using $a_j\le La_{j-1}$ we have
\begin{align*}
c_j=(a_j)^{\ez}b_j\le L^{\ez} (a_{j-1})^{\ez}
\frac{L}{L+1}b_{j-1}=\frac{L^{1+\ez}}{L+1}c_{j-1}.
\end{align*}
Since $L>2$, we can choose a suitable $\ez>0$ such that
$2^{-\ez}=L^{1+\ez}/(1+L)$. Therefore
$$c_j\le 2^{-\ez} c_{j-1}.$$

\noindent{\it Case 3}: If $a_j\ge \frac 12 a_{j-1}$ and $a_j\le\max\{2^{-j},d_{j-1}\}$, by $d_{j-1}\le M2^{-j}$ we have
$$a_j\le \max\{2^{-j},d_{j-1}\}\le (M+1)2^{-j}.$$
Since $b_j\le M 2^{-j}\le M$, we obtain
\begin{align*}
c_j=(a_j)^{\ez}b_j\le (M+1)^{1+\ez}(2^{-\ez})^j\le
(M+1)^{1+\ez}(2^{-\ez})^j
+2^{-\ez} c_{j-1}.
\end{align*}

Combing above all cases, we conclude that
\begin{align*}
c_j\le 2^{-\ez} c_{j-1}+(M+1)^{1+\ez}(2^{-\ez})^j,\quad \forall j\ge j_0+1.
\end{align*}
By iteration, one has
\begin{align*}
c_j&\le (2^{-\ez})^{j-j_0}c_{j_0}+(M+1)(j-j_0) (2^{-\ez})^{j}
\le (M+1)^{1+\ez} (2^{-\ez})^j[2^{\ez j_0}+j],\quad \forall j\ge j_0+1.
\end{align*}
In view of \eqref{it-x1}, $a_{j_0}\le M$ and
$b_{j_0}\le M$, we deduce
\begin{align*}
(b_{j+1})^{1+\ez}\le L^{\ez} (a_{j})^{\ez} b_j
=L^{\ez} c_j\le [(L+1)(M+1)]^{1+\ez} (2^{-\ez})^j[2^{\ez j_0}+j].
\end{align*}
Now set $\theta=2^{-\ez/(1+\ez)}\in (1/2,1)$. Thank to this, noting that
$(2^{\ez j_0}+j)^{\frac 1{1+\ez}}\le 2^{j_0 \theta}+j$, it follows that
\begin{align*}
b_{j+1}&\le C(L,j_0)(M+1)(\theta^j+j\theta^j).
\end{align*}
Hence we finish this proof.

\end{proof}

\section{Proof of Theorem \ref{th1} and Theorem \ref{th2}}

We now proof Theorem \ref{th1} and Theorem \ref{th2} in this section.

\begin{proof}[Proof of Theorem \ref{th1}]
We begin by assuming that $3\le n\le 9$. Indeed, in the case $n=2$,
one can add extra artificial variables (see for instance \cite{cfrs}). Given any
$x_0\in \Omega$, let $R>0$ satisfy $R<\frac{1}{4}{\rm dist}(x_0,\partial \Omega)$ and
hence $B(x_0,R)\Subset \Omega$. We may suppose that
$x_0=0$ and $R=1$ by translation and scaling.
We claim that there exists a dimensional constant $\az=\az(n)\in (0,1)$ such that
\begin{align}\label{cla}
\int_{B_{\rho}}(Du\cdot x)^2|x|^{-n}\,dx\le C\rho^{2\az}
\quad \forall0<\rho<1/2,
\end{align}
where $C=C(A,n,K)(\|u\|^2_{W^{1,2}(B_1)}+1)$. Suppose that this claim holds for
the moment. Then \eqref{holder} follows by \cite[Theorem C.2]{c22}.

We split the proof of this claim in two steps.

\noindent{\bf Step} 1:
We prove that there exists a constant $\theta=\theta(n)\in (0,1/2)$ such
that
\begin{align}\label{cla1}
\int_{B_{2^{-j-1}}}(Du\cdot x)^2|x|^{-n}\,dx \le
C(\theta^j +j\theta^j)\quad\forall j\ge 1,
\end{align}
where $C=C(A,n,K)(\|u\|^2_{W^{1,2}(B_1)}+1)$.

 Let $j_0\ge 1$ be a universal constant and let $\theta=\theta(n)\in (0,1/2)$ to be
chosen later.
  Observe that, by Lemma
\ref{key-le} we clearly have
\begin{align*}
\int_{B_{2^{-j}}}(Du\cdot x)^2|x|^{-n}\,dx\le C(n)\int_{B_1}|Du|^2\,dx
\le C(n)\theta^{-j_0-1}  \theta^j\int_{B_1}|Du|^2\,dx\quad\forall 1\le j\le j_0+1.
\end{align*}
Below we consider the case $j\ge j_0+1$ for \eqref{cla1}.

Let $j\ge j_0+1$ and set
\begin{align*}
a_j:=(2^{-j})^{-n+2}\int_{B_{2^{-j}}}|Du|^2\,dx,\quad
b_j:=\int_{B_{2^{-j}}}(Du\cdot x)^2|x|^{-n}\,dx,\ d_j:=(2^{-j})^{-n+3}
\int_{B_{2^{-j}}}u^2\,dx.
\end{align*}
Clearly, we have
\begin{align*}
a_{j}=(2^{-j})^{-n+2}\int_{B_{2^{-j}}}|Du|^2\,dx
\le 2^{n-2}(2^{-j+1})^{-n+2}\int_{B_{2^{-j+1}}}|Du|^2\,dx=2^{n-2}a_{j-1}.
\end{align*}
By Lemma \ref{key-le} we also get
\begin{align*}
b_j=\int_{B_{2^{-j}}}(Du\cdot x)^2|x|^{-n}\,dx
\le C(n)\int_{B_{\frac 32\times2^{-j}}\backslash B_{2^{-j}}}|Du|^2|x|^{-n+2}\,dx
\le C(n)a_{j-1}.
\end{align*}
On the other hand, applying $p=n>2$ to Lemma \ref{mo} one has
\begin{align*}
d_{j}=(2^{-j})^{-n+3}
\int_{B_{2^{-j}}}u^2\,dx\le C(n)\|u\|^2_{W^{1,2}(B_1)}2^{-j}.
\end{align*}
Combing above we conclude that
\begin{align}\label{ft}
b_j\le b_{j-1},\quad a_j+b_j\le C_0(n)a_{j-1},\
d_j\le C(n)\|u\|^2_{W^{1,2}(B_1)}2^{-j}\quad \forall j\ge j_0+1.
\end{align}

Now using Lemma \ref{key-it} with $\dz=\frac 12$, we can find a sufficiently large $j_0\ge1$ independent of $u$ and $f$ such that if
 \begin{align*}
 a_j\ge\frac 12 a_{j-1}\quad {\rm and}\quad a_j\ge \max\{2^{-j},d_{j-1}\},\quad \forall j\ge j_0+1,
 \end{align*}
then there exists a constant $C(n,A,K)$ such that
 \begin{align}\label{t4}
 (2^{-j})^{-n+2}\int_{B_{\frac32\times 2^{-j}}}|Du|^2\,dx
 \le C(n,A,K)\int_{B_{\frac 32\times 2^{-j}}\backslash B_{2^{-j}}}
 (Du\cdot x)^2|x|^{-n}\,dx.
 \end{align}
Thank to this, it follows from Lemma \ref{key-le} that
\begin{align*}
b_j
\le C(n)(2^{-j})^{-n+2}\int_{B_{\frac 32\times 2^{-j}}}
|Du|^2\,dx&\le C(n,A,K)\int_{B_{\frac32\times 2^{-j}}\backslash B_{2^{-j}}}
 (Du\cdot x)^2|x|^{-n}\,dx\\
 &\le C(n,A,K)(b_{j-1}-b_j).
\end{align*}
Hence  we conclude that there exists a universal constant $j_0\ge 1$ independent of
$f$ and $u$ such that if
 \begin{align}\label{as}
 a_j\ge \frac 12 a_{j-1}\quad {\rm and}\quad a_j\ge \max\{2^{-j},d_{j-1}\}
 \end{align}
 for all $j\ge j_0+1$, one has
 \begin{align}\label{t5}
 b_j\le C(n,A,K)(b_{j-1}-b_j).
 \end{align}
Now by writing $L:=C_0(n)+C(n,A,K)$ and $M:=C(n)(\|u\|^2_{W^{1,2}(B_1)}+1)$.

Combing \eqref{ft} and \eqref{as} with \eqref{t5}, we conclude that
 \begin{enumerate}
 \item[$\bullet$] $b_j\le b_{j-1}$ for all $j\ge j_0+1$;
  \item[$\bullet$] $a_j+b_j\le La_{j-1}$ for all $j\ge j_0+1$;
  \item[$\bullet$]$d_{j}\le M2^{-j}$ for all $j\ge j_0+1$;
  \item[$\bullet$] If $a_j\ge \frac 12a_{j-1}$ and $a_j\ge \max\{2^{-j},d_{j-1}\}$,
  then $b_j\le L(b_j-b_{j-1})$ for all $j\ge j_0+1$.
 \end{enumerate}

Thank to this, by Lemma \ref{key-le} and \cite[Lemma 1.7]{pzz22} we also have $
a_{j_0},b_{j_0},d_{j_0}\le M$, hence one can use Lemma \ref{it} to get
\eqref{cla1} as desired.

\indent{\bf Step 2}: We proof \eqref{cla}.
Given any $0<\rho\le 1/2$, we can find
$k\ge 1$ such that
$2^{-k-1}\le \rho\le2^{-k}$.
Note that $\theta\in (1/2,1)$. We write $\tau_\theta=\theta^{-1}\in (1,2)$ since
$\theta\in (\frac 12,1)$. Then,
\begin{align}\label{t1}
k\le -\log_{\tau_\theta} \rho/\log_{\tau_\theta}2 \le k+1.
\end{align}
Since $0<\rho\le 1/2$, from \eqref{t1} one has
\begin{align}\label{t2}
k\ge -\log_{\tau_\theta} \rho/\log_{\tau_\theta}2-1\ge 0.
\end{align}
Then by $\rho\le2^{-k}$ and \eqref{cla1} we get
\begin{align}\label{t3}
\int_{B_\rho}(Du\cdot x)^2|x|^{-n}\,dx\le C
(\theta^k+k\theta^k),
\end{align}
where $C=C(n,A,K)(\|u\|^2_{W^{1,2}(B_1)}+1)$.
Note that $\tau_{\theta}^{\log_{\tau_\theta} \rho}=\rho$ and $\tau_\theta=\theta^{-1}\in (1,2)$, it follows by \eqref{t2} that
\begin{align*}
\theta^k\le \theta^{-1}\left[(\theta^{-1})^{\log_{\tau_\theta} \rho}\right]^{1/\log
_{\tau_\theta} 2}
= \theta^{-1} \rho^{4\az},
\end{align*}
where $4\az=1/\log_{\tau_\theta} 2\in (0,1)$ since $\theta\in (1/2,1)$. Thus we get
\begin{align*}
\int_{B_\rho}(Du\cdot x)^2|x|^{-n}\,dx\le
C(\rho^{4\az}+\rho^{4\az}\log_{\tau_\theta} \rho^{-1}),\quad \forall 0<\rho
\le 1/2,
\end{align*}
where $C=C(n,A,K)(\|u\|^2_{W^{1,2}(B_1)}+1)$.
Observe that
$$0\le \rho^{2\az}\log_{\tau_\theta} \rho^{-1}\le \theta^{\frac 1{\ln (1/\theta)}}
\frac 1{\az\ln (1/\theta)}\le\frac 1\az ,\quad \forall \rho\in [0,1],$$
since the function $\rho^{\az}\log_{\tau_\theta}  \rho^{-1}$ attains its
maximum at $\rho=\theta^{\frac 1{\az\ln (1/\theta)}}$ on $[0,1]$.
Hence we complete this proof.

\end{proof}

Thanks to Theorem \ref{th1}, we are now ready to prove Theorem \ref{th2}.

\begin{proof}[Proof of Theorem \ref{th2}]
Let $2\le n\le 9$. We first approximate the $W^{1,2}$-stable solution
via $C^2$-stable solution.  Given any $\ez\in (0,1)$ and $x_0\in \Omega$, we claim that there exists
$r_0=r_0(A,K,\Omega)\in (0,\frac 14{\rm dist}(x_0,\partial \Omega))$ such that
$u^{\ez}\in C^2(B(x_0,r_0)$ is a stable solution to
\begin{align}\label{app}
-\bdz u^{\ez}=f_{\ez}(u^{\ez})\quad{\rm in}\ B_{r_0}(x_0);\
u^{\ez}=u\quad{\rm on}\ \partial B_{r_0}(x_0),
\end{align}
where $f_{\ez}\in{\rm Lip}_{\loc}(\rr)$ is given by
\begin{align*}
f_{\ez}(t):=\left\{
\begin{aligned}
&f(t)\quad&{\rm if}\ t<1/\ez ,\\
&f(1/\ez )+f'_-(1/\ez )(t-1/\ez )\quad&{\rm if}\ t\ge 1/\ez.
\end{aligned}
\right.
\end{align*}
Moreover, $u^{\ez}\to u$ in $W^{1,2}(B_{r_0}(x_0))$ and
$u^{\ez}\to u$ a.e. in $B_{r_0}(x_0)$ as $\ez\to 0$.

%
To see this, we follow the idea of \cite{cfrs,pzz22}. Define
$$\bar A:=\min\{f'_{-}(0),A\}.$$
Since $f$ is nondecreasing, we have
\begin{align}\label{dA}
0\le \bar A\le A.
\end{align}
Moreover, we conclude that
\begin{align}\label{clA}
f(t)\ge \bar At-K\quad \forall t\in \rr.
\end{align}
Recalling that $f(t)\ge A\min\{0,t\}-K$ for all $t\in \rr$(by the assumption
\eqref{as-f}). Hence
$f(0)\ge -K$. So by the convexity of
$f$, we get
\begin{align}\label{clA-1}
f(t)\ge f'_{-}(0)t+f(0)\ge f'_{-}(0)t-K\quad \forall t\in \rr.
\end{align}
Noting that $f'_{-}(0)\ge 0$. If $f'_{-}(0)\le A$, \eqref{clA} follows. If
$f'_{-}(0)\ge A$, by \eqref{as-f} note that
$$f(t)\ge A\min\{0,t\}-K=At-K\quad \forall t<0.$$
It remains to check $f(t)\ge At-K$ for all $t\ge 0$. This follows by
$f'_{-}(0)\ge A\ge 0$ and \eqref{clA-1} as desired.

Since $0\le \bar A\le A$, we can choose $r^{\star}>0$ satisfying
$(C_0(n)r^{\star})^{-1}>8A$ for some
dimensional constant $C_0(n)>0$ such that $\bar A$ is not first eigenvalue of
problem $-\bdz v=\lz v$ in $B_r(x_0)$ with $v=0$ on
$\partial B_{r}(x_0)$ for all $r<r^{\star}$ and for some $\lz>0$. Hence, given $r_0<\min\{r^{\star},\frac 12{\rm dist}
(x_0,\partial \Omega)\}$, by \eqref{clA} and \cite[Lemma 3.1]{pzz22}
we conclude that there exists a solution $v\in W^{1,2}(B_{r_0}(x_0))\cap C^2(B_{r_0}(x_0))$ to
$$-\bdz v=\bar Av-K\quad{\rm in}\ B_{r_0}(x_0);\ v=u\quad{\rm on}\ \partial
B_{r_0}(x_0)$$
Moreover, we have
$$v\le u\quad {\rm a.e.\ in}\quad B_{r_0}(x_0).$$
From this, noting that $f_{\ez}\in {\rm Lip}_{\loc}(\rr)$ is nondecreasing and convex satisfying
\eqref{clA},
one can repeat the argument of \cite{cfrs,pzz22} to establish a $C^2$-stable solution
$u^{\ez}$ in approximating $W^{1,2}$-stable solution $u$. Here we omit details since this proof follows
by a slight modifications; see \cite[Section 3]{pzz22}.

Observe that, by \eqref{dA} and \eqref{clA}, all constants in Lemma \ref{sob} only depend on
$A$, $K$ and $n$, so applying Theorem \ref{th1} to $u^{\ez}$, we obtain
\begin{align*}
\underset{{B_r(x_0)}}{\rm osc}\ u^{\ez}\le C(n,A,K)\left(\frac{r}R\right)^{\alpha}(\|u^{\ez}\|_{W^{1,2}(B_{2R}(x_0))}+1),\quad
\forall x_0\in\Omega,\, 0<r<R<\frac1 4r_0,
\end{align*}
and hence passing to limit $\ez\to 0$ as desired.
\end{proof}

\section{Appendix}
Here, we provide the proof of Lemma \ref{mo}. We first recall the following key result
due to Cabr\'e, Figalli, Ros-Oton and Serra \cite[Lemma 2.1]{cfrs}.
\begin{lem}\label{key-le}
Let $n\ge 2$ and $f\in {\rm Lip}_{\loc}(\rr)$. Assume
that $u\in C^2(B_2)$ is a stable solution of $-\bdz u=f(u)$ in $B_2$.
Then for all $\eta\in C^{0,1}_c(B_2)$, we have
\begin{align}\label{keq-1}
&(n-2)\int_{B_2}|Du|^2\eta^2\,dx\nonumber\\
&\le -2\int_{B_2}(x\cdot D\eta)\eta|Du|^2+4\int_{B_2}(x\cdot Du)(Du\cdot D\eta)\eta\,dx+ \int_{B_2}(Du\cdot x)^2|D\eta|^2\,dx.
\end{align}
Moreover, if $3\le n\le 9$, then for all $r\in (0,1)$ it holds
\begin{align}\label{keq-2}
\int_{B_r}(Du\cdot x)^2|x|^{-n}\,dx\le C(n)\int_{B_{3r/2}\backslash B_r}
|Du|^2|x|^{-n+2}\,dx.
\end{align}

\end{lem}
\begin{proof}[Proof of Lemma \ref{mo}]
For all $\eta\in C^{0,1}_c(B_2)$, by Lemma \ref{key-le} we obtain
\begin{align}\label{keq}
&(n-2)\int_{B_2}|Du|^2\eta^2\,dx\nonumber\\
&\le -2\int_{B_2}(x\cdot D\eta)\eta)|Du|^2+4\int_{B_2}(x\cdot Du)(Du\cdot D\eta)\eta\,dx+ \int_{B_2}(Du\cdot x)^2|D\eta|^2\,dx.
\end{align}
For any $0<r<1$, we take a test function $\eta\in C^{0,1}_c(B_2)$ such that
$\eta=r^{-\frac{n-2}{2}}$ in $B_r$ and $\eta=|x|^{-\frac{n-2}{2}}\phi$ in
$B_2\backslash \overline B_r$,
where $\phi\in C^\fz_c(B_2)$ satisfies
\begin{align}\label{cut-1}
\phi=1\ {\rm on}\ B_{3/2},\ 0\le \phi\le 1\ {\rm on}\ B_2,\ {\rm and}\
|D\phi|\le 8\ {\rm on}\ B_2.
\end{align}
Since
\begin{align*}
D\eta=0\ {\rm on}\ B_r,\quad\ D\eta=-\frac{n-2}{2}|x|^{-\frac {n-2}2-2}x \phi
+D\phi|x|^{-\frac{n-2}2}\ {\rm on}\ B_2\backslash B_r,
\end{align*}
then substituting this in \eqref{keq} we get
\begin{align*}
&(n-2)r^{-n+2}\int_{B_r}|Du|^2\,dx+
\frac{(n-2)(10-n)}{4}\int_{B_2\backslash B_r}(Du\cdot x)^2|x|^{-n}\phi^2\,dx\nonumber\\
&\le -2\int_{B_2\backslash B_r}|x|^{-2+n}|Du|^2\phi(x\cdot D\phi)\,dx
+4\int_{B_2\backslash B_r}|x|^{-n+2}
(x\cdot Du)\phi (Du\cdot D\phi)\,dx\\
&\quad+(2-n)\int_{B_2\backslash B_r}(Du\cdot x)^2|x|^{-n}\phi(x\cdot D\phi)\,dx
+\int_{B_2\backslash B_r}|x|^{-n+2}(x\cdot Du)^2|D\phi|^2\,dx.
\end{align*}
Note that $D\phi=0$ on $B_{3/2}$ and  $\|\phi\|_{W^{1,\fz}(B_2)}\le 10$ by \eqref{cut-1},
then the right-hand side of this inequality above can be bounded by
$C(n)\int_{B_2\backslash B_{3/2}}|Du|^2\,dx$. On the other hand,
since $3\le n\le 9$, the second term of left-hand side is
nonnegative. Thus we deduce that
\begin{align*}
r^{-n+2}\int_{B_r}|Du|^2\,dx\le C(n)\int_{B_2}|Du|^2\,dx\quad \forall 0<r\le 1.
\end{align*}
Then by  Poincar\'e inequality and H\"older inequality we have
\begin{align*}
\mint{-}_{B_r}\left|u-\mint{-}_{B_r}u\,dy\right|\,dx
&\le C(n)\left(r^{-n+2}\int_{B_r}|Du|^2\,dx\right)^{\frac12}
\le C(n)\|Du\|_{L^2(B_2)}\quad\forall 0<r<1,
\end{align*}
which further leads to
\begin{align}\label{bmo}
\|u\|_{{\rm BMO}(B_1)}\le C(n)\|Du\|_{L^2(B_2)}.
\end{align}
Hence using the equivalence of a norm on BMO space (\cite[Corollary 6.12]{d01}) one gets
\begin{align*}
\int_{B_1}\left|u-\mint{-}_{B_1}u\,dy\right|^p\,dx\le C(n,p)\|u\|^p_{{\rm BMO}(B_1)}\quad \forall p>2.
\end{align*}
Thank to this, via H\"older inequality and \eqref{bmo} we obtain
\begin{align*}
\int_{B_1}|u|^p\,dx\le \int_{B_1}\left|u-\mint{-}_{B_1}u\,dy\right|^p\,dx
+\left(\mint{-}_{B_1}|u|\,dx\right)^p\le C(n,p)\|u\|^p_{W^{1,2}(B_2)}.
\end{align*}
From this, now by H\"older inequality with $p>2$ we get
\begin{align*}
\int_{B_r}u^2\,dx\le C(n,p)\left(\int_{B_r}|u|^p\,dx\right)^{\frac 2p} r^{\frac{n(p-2)}{p}}
\le C(n,p)\|u\|^2_{W^{1,2}(B_2)}r^{\frac{n(p-2)}{p}}\quad \forall 0<r<1.
\end{align*}
This proof is complete.

\end{proof}

\section*{Funding}
The author is supported by National Natural Science of
Foundation of China (No. 12201612 \& No. 11688101) and by China Postdoctoral Science Foundation funded project (No. BX20220328). The author is also supported by National
 key R \& D Program of China (No. 2021 YFA 1003100).

\section*{Acknowledgments}
The author would like to thank Professor Xavier Cabr\'e for pointing out
that the constant $C$ in Lemma \ref{key-it} should depend on
$A$ and $K$. Also the author would like to thank Professor Yi Ru-Ya Zhang and Yuan Zhou for many helpful comments and valuable discussions related to the subject.
In particular, Professor Yi Ru-Ya Zhang gives several important suggestions
on proof of Lemmas \ref{key-it} and \ref{it}. The author would also like to thank  the anonymous referee for several valuable comments and suggestions in the previous version of this paper.

\noindent  Fa Peng

\noindent
Academy of Mathematics and Systems Science, the Chinese Academy of Sciences, Beijing 100190, P. R. China

\noindent{\it E-mail }:  \texttt{fapeng@amss.ac.cn}

\end{document}